\documentclass[11pt,reqno]{amsart}
 \usepackage{amssymb,amsfonts,amscd,amsbsy}
\usepackage[mathscr]{eucal}
\usepackage{url}

\newtheorem{thm}{Theorem}[section]
\newtheorem{lem}[thm]{Lemma}
\newtheorem{prop}[thm]{Proposition}
\newtheorem{cor}[thm]{Corollary}
\theoremstyle{definition}

\numberwithin{equation}{section}

\makeatletter
\def\imod#1{\allowbreak\mkern5mu({\operator@font mod}\,\,#1)}
\makeatother

\begin{document}

\title[$q$-hypergeometric double sums as mock theta functions]
{$q$-hypergeometric double sums as mock theta functions} 
 
\author{Jeremy Lovejoy and Robert Osburn}

\address{CNRS, LIAFA, Universit{\'e} Denis Diderot - Paris 7, Case 7014, 75205 Paris Cedex 13, FRANCE}

\address{School of Mathematical Sciences, University College Dublin, Belfield, Dublin 4, Ireland}

\email{lovejoy@liafa.jussieu.fr}

\email{robert.osburn@ucd.ie}


\date{\today}

\begin{abstract}
Recently, Bringmann and Kane established two new Bailey pairs and used them to relate certain $q$-hypergeometric series to real quadratic fields. We show how these pairs give rise to new mock theta functions in the form of $q$-hypergeometric double sums.  Additionally, we prove an identity between one of these sums and two classical mock theta functions introduced by Gordon and McIntosh.
\end{abstract}
 
\dedicatory{In memory of Basil Gordon}

\maketitle

\section{Introduction}
A \emph{Bailey pair} relative to $a$ is a pair of sequences $(\alpha_n,\beta_n)_{n \geq 0}$ satisfying

\begin{equation} \label{pairdef}
\beta_n = \sum_{k=0}^n \frac{\alpha_k}{(q)_{n-k}(aq)_{n+k}}. 
\end{equation} 
Here we have used the standard $q$-hypergeometric notation, 

\begin{equation*}
(a)_n = (a;q)_n = \prod_{k=1}^{n} (1-aq^{k-1}),
\end{equation*}
valid for $n \in \mathbb{N} \cup \{\infty\}$.
The \emph{Bailey lemma} states that if $(\alpha_n,\beta_n)$ is a Bailey pair relative to $a$, then so is $(\alpha'_n,\beta'_n)$, where

\begin{equation} \label{alphaprimedef}
\alpha'_n = \frac{(b)_n(c)_n(aq/bc)^n}{(aq/b)_n(aq/c)_n}\alpha_n
\end{equation} 

\noindent and

\begin{equation} \label{betaprimedef}
\beta'_n = \sum_{k=0}^n\frac{(b)_k(c)_k(aq/bc)_{n-k} (aq/bc)^k}{(aq/b)_n(aq/c)_n(q)_{n-k}} \beta_k.
\end{equation}
Inserting \eqref{alphaprimedef} and \eqref{betaprimedef} into \eqref{pairdef} with $n \to \infty$ gives

\begin{equation} \label{limitBailey}
\sum_{n \geq 0} (b)_n(c)_n (aq/bc)^n \beta_n = \frac{(aq/b)_{\infty}(aq/c)_{\infty}}{(aq)_{\infty}(aq/bc)_{\infty}} \sum_{n \geq 0} \frac{(b)_n(c)_n(aq/bc)^n }{(aq/b)_n(aq/c)_n}\alpha_n,
\end{equation} 

\noindent valid whenever both sides converge.  For more on Bailey pairs, including historical perspectives and recent advances, see Chapter 3 of \cite{An2}, \cite{An3},  or \cite{war}.

In a recent study of multiplicative $q$-series, Bringmann and Kane \cite{bk} established two new and interesting Bailey pairs.   They showed that $(a_n,b_n)$ is a Bailey pair relative to $1$, where

\begin{equation} \label{a2n}
a_{2n} = (1-q^{4n})q^{2n^2-2n}\sum_{j=-n}^{n-1}q^{-2j^2-2j},
\end{equation}

\begin{equation} \label{a2n+1}
a_{2n+1} = -(1-q^{4n+2})q^{2n^2}\sum_{j=-n}^{n} q^{-2j^2},
\end{equation}

\noindent and

\begin{equation} \label{bn}
b_n = \frac{(-1)^n(q;q^2)_{n-1}}{(q)_{2n-1}} \chi(\text{$n \neq 0$}),
\end{equation}

\noindent and $(\alpha_n,\beta_n)$ is a Bailey pair relative to $q$, where

\begin{equation} \label{alpha2n}
\alpha_{2n} = \frac{1}{1-q}\left(q^{2n^2+2n}\sum_{j=-n}^{n-1}q^{-2j^2-2j} + q^{2n^2}\sum_{j=-n}^{n} q^{-2j^2}\right),
\end{equation}

\begin{equation} \label{alpha2n+1}
\alpha_{2n+1} = -\frac{1}{1-q}\left(q^{2n^2+4n+2}\sum_{j=-n}^{n} q^{-2j^2} + q^{2n^2+2n}\sum_{j=-n-1}^n q^{-2j^2-2j}\right),
\end{equation}

\noindent and 

\begin{equation} \label{betan}
\beta_n = \frac{(-1)^n(q;q^2)_n}{(q)_{2n+1}}.
\end{equation}

These closely resemble Bailey pairs related to $7$th order mock theta functions \cite{An1.5}, but surprisingly no $q$-series obtained by a direct substitution of either \eqref{a2n}--\eqref{bn} or \eqref{alpha2n}--\eqref{betan} in \eqref{limitBailey} is a genuine mock theta function.  For example, it turns out that substituting \eqref{a2n}--\eqref{bn} in \eqref{limitBailey} with $b$, $c \to \infty$ yields 

\begin{equation*}
\frac{-q}{(-q)_{\infty}} \omega(q)
\end{equation*}

\noindent where $\omega(q)$ is one of the third order mock theta functions.  The presence of the infinite product means that this is not a mock theta function but a \emph{mixed mock modular form}.  

Recall that mock theta functions are $q$-series which were introduced by Ramanujan in his last letter to G. H. Hardy on January 12, 1920. Until 2002, it was not known how these functions fit into the theory of modular forms. Thanks to work of Zwegers \cite{Zw1} and Bringmann and Ono \cite{Br-On1,Br-On2}, we now know that each of Ramanujan's examples of mock theta functions is the holomorphic part of a weight $1/2$ harmonic weak Maass form $f(\tau)$ (as usual, $q := e^{2 \pi i \tau}$ where $\tau=x + iy \in \mathbb{H}$). Following Zagier \cite{Za1}, the holomorphic part of any weight $k$ harmonic weak Maass form $f$ is called a mock modular form of weight $k$.  If $k=1/2$ and the image of $f$ under the operator $\xi_k := 2iy^k \frac{\overline{\partial}}{\partial \overline{\tau}}$ is a unary theta function, then the holomorphic part of $f$ is called a mock theta function.    Specializations of the Appell-Lerch series
\begin{equation*} 
m(x,q,z) := \frac{1}{j(z,q)} \sum_{r \in \mathbb{Z}} \frac{(-1)^r q^{\binom{r}{2}} z^r}{1-q^{r-1} xz}
\end{equation*}
are perhaps the most well-known and most important class of mock theta functions \cite{Za1,Zw1}.  Here $x$, $z \in \mathbb{C}^{*}:=\mathbb{C} \setminus \{ 0 \}$ with neither $z$ nor $xz$ an integral power of $q$, and 
$$
j(x,q):=(x)_{\infty} (q/x)_{\infty} (q)_{\infty}.
$$
For more on mock theta functions, their remarkable history and modern developments, see \cite{On1} and \cite{Za1}.

The goal of this paper is to obtain genuine mock theta functions from the Bailey pairs of Bringmann and Kane by first moving a step along the Bailey chain. Applying \eqref{alphaprimedef} and \eqref{betaprimedef} to $(a_n,b_n)$ with $(b,c) \to (-1,\infty)$ and to $(\alpha_n,\beta_n)$ with $(b,c) \to (-q,\infty)$, we obtain the Bailey pairs recorded in the following two lemmas.

\begin{lem} \label{lemma1}
The pair $(a_n',b_n')$ is a Bailey pair relative to $1$, where

\begin{equation*} \label{a2nprime}
a_{2n}' = 2(1-q^{2n})q^{4n^2-n}\sum_{j=-n}^{n-1}q^{-2j^2-2j},
\end{equation*}

\begin{equation*} \label{a2n+1prime}
a_{2n+1}' = -2(1-q^{2n+1})q^{4n^2+3n+1}\sum_{j=-n}^{n} q^{-2j^2},
\end{equation*}

\noindent and

\begin{equation*} \label{bnprime}
b_n' = \frac{1}{(-q)_n}\sum_{j=1}^n \frac{(-1)_j(q;q^2)_{j-1}(-1)^jq^{\binom{j+1}{2}}}{(q)_{n-j}(q)_{2j-1}}.
\end{equation*}
\end{lem}

\begin{lem} \label{lemma2}
The pair $(\alpha_n',\beta_n')$ is a Bailey pair relative to $q$, where

\begin{equation*} \label{alpha2nprime}
\alpha_{2n}' = \frac{1}{1-q}\left(q^{4n^2+3n}\sum_{j=-n}^{n-1}q^{-2j^2-2j} + q^{4n^2+n}\sum_{j=-n}^{n} q^{-2j^2}\right),
\end{equation*}

\begin{equation*} \label{alpha2n+1prime}
\alpha_{2n+1}' = -\frac{1}{1-q}\left(q^{4n^2+7n+3}\sum_{j=-n}^{n} q^{-2j^2} + q^{4n^2+5n+1}\sum_{j=-n-1}^n q^{-2j^2-2j}\right),
\end{equation*}

\noindent and 

\begin{equation*} \label{betanprime}
\beta_n' = \frac{1}{(-q)_n} \sum_{j=0}^n \frac{(-q)_j(q;q^2)_j(-1)^jq^{\binom{j+1}{2}}}{(q)_{n-j}(q)_{2j+1}}.
\end{equation*}
\end{lem}

With our main result, we present four mock theta functions arising from the Bailey pairs in Lemmas \ref{lemma1} and \ref{lemma2}.   Define 
\begin{equation*}
\begin{aligned}
& \theta_{n,p}(x,y,q) := \frac{1}{\overline{J}_{0, np(2n+p)}} \sum_{r^{*} = 0}^{p-1} \sum_{s^{*}=0}^{p-1} q^{n\binom{r-(n-1)/2}{2} + (n+p)(r - (n-1)/2)(s+ (n+1)/2) + n \binom{s + (n+1)/2}{2}} \\
& \times \frac{(-x)^{r - (n-1)/2} (-y)^{s + (n+1)/2} J_{p^{2} (2n+p)}^{3} j(-q^{np(s-r)} x^{n} / y^{n}, q^{np^2}) j(q^{p(2n+p)(r+s) + p(n+p)} x^{p} y^{p}, q^{p^2 (2n + p)})}{j(q^{p(2n+p)r + p(n+p)/2} (-y)^{n+p} / (-x)^{n}, q^{p^{2} (2n+p)}) j(q^{p(2n+p)s + p(n+p)/2} (-x)^{n+p} / (-y)^{n}, q^{p^{2} (2n+p)})},
\end{aligned}
\end{equation*}

\noindent where $r := r^{*} + \{(n-1)/2 \}$ and $s:= s^{*} + \{ (n-1)/2 \}$ with $0 \leq \{ \alpha \} < 1$ denoting the fractional part of $\alpha$. Also, $J_{m}:= J_{m, 3m}$ with $J_{a,m} := j(q^{a}, q^{m})$, and $\overline{J}_{a,m}:=j(-q^{a}, q^{m})$.

\begin{thm} \label{main} The following are mock theta functions:

\begin{eqnarray} 
\mathcal{W}_{1}(q) &:=& \sum_{n \geq j \geq 1} \frac{(-1)_j (q; q^2)_{j-1} (-1)^{j} q^{n^2 + \binom{j+1}{2}}}{(-q)_n (q)_{n-j} (q)_{2j-1}} \label{m1} \\ 
&=& 4 m(-q^{17}, q^{48}, -1) - 4q^{-5} m(-q, q^{48}, -1) - \frac{2q^2 \theta_{3,2}(q^5, q^5, q)}{j(q,q^3)}, \nonumber \\
\mathcal{W}_{2}(q) &:=& \sum_{n \geq j \geq 1} \frac{(q; q^2)_n (-1)_{j} (q; q^2)_{j-1} (-1)^{n+j} q^{\binom{j+1}{2}}}{(-q)_{n} (q)_{n-j} (q)_{2j-1}} \label{m2} \\ &=& 4 m(-q, q^8, -1) + \frac{2q \theta_{1,2}(-q^2, -q^2, q)}{j(-1,q)}, \nonumber \\
\mathcal{W}_{3}(q) &:=& \sum_{n \geq j \geq 1} \frac{(q; q^2)_{n} (-1)_j (q^2; q^4)_{j-1} (-1)^{n+j} q^{n^2 + j^2 + j}}{(-q^2; q^2)_{n} (q^2; q^2)_{n-j} (q^2; q^2)_{2j-1}} \label{m4} \\ &=& 4m(-q, q^{12}, -1) + \frac{2q^3 \theta_{1,1}(-q^7, -q^7, q^4)}{j(-q, q^4)}, \nonumber \\
\mathcal{W}_{4}(q)&:=& \sum_{n \geq j \geq 0} \frac{(-q)_{j} (q; q^2)_{j} (-1)^j q^{n^2 + n + \binom{j+1}{2}}}{(-q)_n (q)_{n-j} (q)_{2j+1}}  \label{m6} \\ &=& -2q^{-4} m(-q^5, q^{48}, -1) - 2q^{-2} m(-q^{11}, q^{48}, -1) + \frac{\theta_{3,2}(q^3, q^3, q)}{j(q, q^3)} \nonumber .
\end{eqnarray}

\end{thm}

It should be noted that the series defining $\mathcal{W}_{2}(q)$ does not converge. However, similar to the sixth order mock theta function $\mu(q)$ \cite{AH}, the sequence of even partial sums and the sequence of odd partial sums both converge. We define $\mathcal{W}_{2}(q)$ as the average of these two values.

To prove Theorem \ref{main} we first use the Bailey machinery to express the $\mathcal{W}_i$ in terms of Hecke-type double sums $f_{a,b,c}(x,y,q)$, where

\begin{equation} \label{fdef}
f_{a,b,c}(x,y,q) : = \sum_{\text{sg}(r) = \text{sg}(s)} \text{sg}(r) (-1)^{r+s} x^r y^s q^{a \binom{r}{2} + brs + c \binom{s}{2}}.
\end{equation}

\noindent Here $x$, $y \in \mathbb{C}^{*}$ and sg$(r):=1$ for $r \geq 0$ and sg$(r):=-1$ for $r <0$.  Then we apply recent results of Hickerson and Mortenson \cite{Hi-Mo1} to express the Hecke-type double sums as Appell-Lerch series $m(x, q, z)$ (up to the addition of weakly holomorphic modular forms).   

We highlight one connection to classical mock theta functions.   Namely, we express the multisum \eqref{m2} in terms of the ``eighth order" mock theta functions $S_1(q)$ and $T_1(q)$, defined by (see \cite{Go-Mc1})

$$
S_1(q) := \sum_{n \geq 0} \frac{q^{n(n+2)} (-q; q^2)_n}{(-q^2; q^2)_n}
$$

\noindent and

$$
T_1(q) := \sum_{n \geq 0} \frac{q^{n(n+1)} (-q^2; q^2)_n}{(-q; q^2)_{n+1}}.
$$

\begin{cor} \label{8th} We have the identity

$$
\mathcal{W}_{2}(q)= 2q T_1(q) -q S_{1}(q).
$$

\end{cor}
\noindent Similar identities involving mock theta functions and multisums were given by Andrews \cite[Section 13]{An4}, and more could be deduced from \cite[Theorem 2.4]{Br-Lo-Os1}.

The paper proceeds as follows.   Some background material on Hecke-type double sums and Appell-Lerch series is collected in Section 2, and Theorem \ref{main} and Corollary \ref{8th} are established in Section 3.

\section{Preliminaries}

We recall some relevant preliminaries. The most important is a result which allows us to convert from the Hecke-type double sums \eqref{fdef} to Appell-Lerch series. Define

\begin{equation} \label{g}
\begin{aligned}
g_{a,b,c}(x, y, q, z_1, z_0) & := \sum_{t=0}^{a-1} (-y)^t q^{c\binom{t}{2}} j(q^{bt} x, q^a) m\left(-q^{a \binom{b+1}{2} - c \binom{a+1}{2} - t(b^2 - ac)} \frac{(-y)^a}{(-x)^b}, q^{a(b^2 - ac)}, z_0) \right)\\
& + \sum_{t=0}^{c-1} (-x)^t q^{a \binom{t}{2}} j(q^{bt} y, q^c) m\left(-q^{c\binom{b+1}{2} - a\binom{c+1}{2} - t(b^2 -ac)}  \frac{(-x)^c}{(-y)^b}, q^{c(b^2 - ac)}, z_1\right).
\end{aligned}
\end{equation}

Following \cite{Hi-Mo1}, we use the term ``generic" to mean that the parameters do not cause poles in the Appell-Lerch series or in the quotients of theta functions.

\begin{thm}{\cite[Theorem 0.3]{Hi-Mo1}} \label{hm} Let $n$ and $p$ be positive integers with $(n$, $p)=1$. For generic $x$, $y \in \mathbb{C}^{*}$

$$
f_{n, n+p, n}(x, y, q) = g_{n, n+p, n}(x, y, q, -1, -1) + \theta_{n,p}(x,y,q).
$$

\end{thm}

We shall also require certain facts about $j(x,q)$, $m(x,q,z)$ and $f_{a,b,c}(x,y,q)$. From the definition of $j(x,q)$, we have

\begin{equation} \label{j1}
j(q^{n} x, q) = (-1)^{n} q^{-\binom{n}{2}} x^{-n} j(x,q)
\end{equation}

\noindent where $n \in \mathbb{Z}$ and

\begin{equation} \label{j2}
j(x,q) = j(q/x, q) = -x j(x^{-1}, q).
\end{equation}

Next, some relevant properties of the sum $m(x,q,z)$ are given in the following (see (2.2b) of Proposition 2.1 and Theorem 2.3 in \cite{Hi-Mo1}).

\begin{prop} \label{mprop} For generic $x$, $z$, $z_0 \in \mathbb{C}^{*}$.

\begin{equation} \label{appell1}
m(x,q,z) = x^{-1} m(x^{-1}, q, z^{-1})
\end{equation}

\noindent and

\begin{equation} \label{appell2}
m(x, q, z) = m(x, q, z_0) + \frac{z_0 J_1^3 j(z / z_{0}, q) j(x z z_{0}, q)}{j(z_0, q) j(z, q) j(xz_0, q) j(xz, q)}.
\end{equation}

\end{prop}

Finally, two important transformation properties of $f_{a,b,c}(x,y,q)$ are given in the following (see Propositions 5.1 and 5.2 in \cite{Hi-Mo1}).

\begin{prop} For $x$, $y \in \mathbb{C}^{*}$,

\begin{equation} \label{fprop1}
\begin{aligned}
f_{a,b,c}(x,y,q) = & f_{a,b,c}(-x^2 q^a, -y^2 q^c, q^4) - x f_{a,b,c}(-x^2 q^{3a}, -y^2 q^{c+2b}, q^4) \\
& -y f_{a,b,c}(-x^2 q^{a+2b}, -y^2 q^{3c}, q^4) + xyq^b f_{a,b,c}(-x^2 q^{3a+2b}, -y^2 q^{3c+2b}, q^4)
\end{aligned}
\end{equation}

\noindent and 

\begin{equation} \label{fprop2}
f_{a,b,c}(x, y, q)= - \frac{q^{a+b+c}}{xy} f_{a,b,c}(q^{2a+b}/ x, q^{2c+b} / y, q).
\end{equation}

\end{prop}

\section{Proof of Theorem \ref{main}}

\begin{proof}[Proof of Theorem \ref{main}]

As the proofs of (\ref{m1})--(\ref{m6}) are similar, we give full details only for (\ref{m1}) and (\ref{m6}).  Recall that the goal is to express each double sum $q$-series in terms of Appell-Lerch series.   For (\ref{m1}), apply Lemma \ref{lemma1} and let $b$, $c \to \infty$ in (\ref{limitBailey}) to obtain

\begin{equation*}
\begin{aligned}
\mathcal{W}_{1}(q)  & = \sum_{n \geq 0} q^{n^2} b_n'(q) \\
& = \frac{1}{(q)_{\infty}} \sum_{n \geq 0} q^{n^2} a_n'(q) \\
& = \frac{1}{(q)_{\infty}} \Biggl( \sum_{n \geq 0} q^{4n^2} a_{2n}'(q) + \sum_{n \geq 0} q^{4n^2 + 4n + 1} a_{2n+1}'(q) \Biggr) \\
& =  \frac{2}{(q)_{\infty}} \Biggl( \sum_{n \geq 0} q^{8n^2 - n} \sum_{j=-n}^{n-1} q^{-2j^2 - 2j} - \sum_{n \geq 0} q^{8n^2 + n} \sum_{j=-n}^{n-1} q^{-2j^2 - 2j} \\
& - \sum_{n \geq 0} q^{8n^2 + 7n + 2} \sum_{j=-n}^{n} q^{-2j^2} - \sum_{n \geq 0} q^{8n^2 + 9n + 3} \sum_{j=-n}^{n} q^{-2j^2} \Biggr). \\
\end{aligned}
\end{equation*}

After replacing $n$ with $-n$ in the second sum and $n$ with $-n-1$ in the fourth sum, we let $n=(r+s+1)/2$, $j=(r-s-1)/2$ in the first two sums and $n=(r+s)/2$, $j=(r-s)/2$ in the latter two sums to find

\begin{equation*}
\begin{aligned}
\mathcal{W}_{1}(q) 
& =  \frac{2q^2}{(q)_{\infty}} \Biggl( \Bigl(  \sum_{\substack{r, s \geq 0 \\ r \not\equiv s \imod{2} }} - \sum_{\substack{r, s < 0 \\ r \not\equiv {s \imod{2}} }} \Bigr)  q^{\frac{3}{2} r^2 + 5rs + \frac{7}{2} r + \frac{3}{2} s^2 + \frac{7}{2} s} \\
& - \Bigl(  \sum_{\substack{r, s \geq 0 \\ r \equiv s \imod{2} }} - \sum_{\substack{r, s < 0 \\ r \equiv {s \imod{2}} }} \Bigr) q^{\frac{3}{2} r^2 + 5rs + \frac{7}{2}r + \frac{3}{2}s^2 + \frac{7}{2}s} \Biggr) \\
& =  \frac{2q^2}{(q)_{\infty}} \Biggl( 2q^{5} f_{3,5,3}(-q^{23}, -q^{19}, q^4) - f_{3,5,3}(-q^{13}, -q^{13}, q^4) - q^{15} f_{3,5,3}(-q^{29}, -q^{29}, q^4) \Biggr) \\
& = -\frac{2q^2}{(q)_{\infty}} f_{3,5,3}(q^5, q^5, q).
\end{aligned}
\end{equation*}
In the penultimate equality, we have replaced $(r,s)$ first by $(2r+1,2s)$ and then by $(2r,2s+1)$ in the first line, and then replaced $(r,s)$ first by $(2r,2s)$ and then by $(2r +1,2s+1)$ in the second line, and then invoked (\ref{fdef}).  In the last step, we have used (\ref{fprop1}). By Theorem \ref{hm}, (\ref{g}), (\ref{j1}) and (\ref{j2}), we have

\begin{equation*}
\begin{aligned}
f_{3,5,3}(q^5, q^5, q) = -2q^{-2} j(q, q^3) m(-q^{17}, q^{48}, -1) + 2q^{-7} j(q, q^3) m(-q, q^{48}, -1) + \theta_{3,2}(q^5, q^5, q)
\end{aligned}
\end{equation*}

\noindent and so

\begin{equation*} 
\begin{aligned}
\mathcal{W}_{1}(q) = 4 m(-q^{17}, q^{48}, -1) - 4q^{-5} m(-q, q^{48}, -1) - \frac{2q^2 \theta_{3,2}(q^5, q^5, q)}{j(q,q^3)}.
\end{aligned}
\end{equation*}

For (\ref{m2}), apply Lemma \ref{lemma1} and let $b=-\sqrt{q}$ and $c=\sqrt{q}$ in (\ref{limitBailey}) to get

\begin{equation*}
\begin{aligned}
\mathcal{W}_{2}(q) & = \sum_{n \geq 0} (-1)^n (q; q^2)_n b_n'(q) \\
& = \frac{(q; q^2)_{\infty}}{2(q^2; q^2)_{\infty}} \sum_{n \geq 0} (-1)^n a_n'(q) \\
& = \frac{(q; q^2)_{\infty}}{2(q^2; q^2)_{\infty}} \Biggl( \sum_{n \geq 0} a_{2n}'(q) - \sum_{n \geq 0} a_{2n+1}'(q) \Biggr) \\
& = \frac{(q; q^2)_{\infty}}{(q^2; q^2)_{\infty}} \Biggl( \sum_{n \geq 0}q^{4n^2 - n} \sum_{j=-n}^{n-1} q^{-2j^2 - 2j} - \sum_{n \geq 0} q^{4n^2 + n} \sum_{j=-n}^{n-1} q^{-2j^2 - 2j} \\
& + \sum_{n \geq 0} q^{4n^2 + 3n + 1} \sum_{j=-n}^{n} q^{-2j^2} - \sum_{n \geq 0} q^{4n^2 + 5n + 2} \sum_{j=-n}^{n} q^{-2j^2} \Biggr).
\end{aligned}
\end{equation*}

\noindent As before, we proceed with

\begin{equation*}
\begin{aligned}
\mathcal{W}_{2}(q) & = \frac{(q; q^2)_{\infty}}{(q^2; q^2)_{\infty}} \Biggl ( \Bigl(  \sum_{\substack{r, s \geq 0 \\ r \not\equiv s \imod{2} }} - \sum_{\substack{r, s < 0 \\ r \not\equiv {s \imod{2}} }} \Bigr) q^{\frac{1}{2}r^2 + 3rs + \frac{3}{2}r + \frac{1}{2}s^2 + \frac{3}{2}s + 1} \\
& +  \Bigl(  \sum_{\substack{r, s \geq 0 \\ r \equiv s \imod{2} }} - \sum_{\substack{r, s < 0 \\ r \equiv {s \imod{2}} }} \Bigr) q^{\frac{1}{2}r^2 + 3rs + \frac{3}{2}r + \frac{1}{2}s^2 + \frac{3}{2}s + 1}  \Biggr) \\
& = \frac{(q; q^2)_{\infty}}{(q^2; q^2)_{\infty}} \Biggl ( \Bigl(  \sum_{r, s \geq 0} - \sum_{r, s < 0} \Bigr) q^{\frac{1}{2}r^2 + 3rs + \frac{3}{2}r + \frac{1}{2}s^2 + \frac{3}{2}s + 1} \Biggr) \\
& =  \frac{q(q; q^2)_{\infty}}{(q^2; q^2)_{\infty}} f_{1,3,1}(-q^2, -q^2, q).
\end{aligned}
\end{equation*}

\noindent By Theorem \ref{hm}, (\ref{g}) and (\ref{j1}), we have

\begin{equation*}
\begin{aligned}
f_{1,3,1}(-q^2, -q^2, q)=2q^{-1} j(-1,q) m(-q, q^8, -1) + \theta_{1,2}(-q^2, -q^2,q)
\end{aligned}
\end{equation*}

\noindent and so

\begin{equation*} 
\begin{aligned}
\mathcal{W}_{2}(q) =4 m(-q, q^8, -1) + \frac{2q \theta_{1,2}(-q^2, -q^2, q)}{j(-1,q)}.
\end{aligned}
\end{equation*}

For (\ref{m4}), apply Lemma \ref{lemma1} and let $b=q$, $c \to \infty$ and $q \to q^2$ in (\ref{limitBailey}) to get

\begin{equation*}
\begin{aligned}
\mathcal{W}_{3}(q) & = \sum_{n \geq 0} (-1)^n (q; q^2)_n q^{n^2} b_{n}'(q^2) \\
& = \frac{(q; q^2)_{\infty}}{(q^2; q^2)_{\infty}} \sum_{n \geq 0} (-1)^n q^{n^2} a_{n}'(q^2) \\
& = \frac{(q; q^2)_{\infty}}{(q^2; q^2)_{\infty}} \Biggl( \sum_{n \geq 0} q^{4n^2} a_{2n}'(q^2) - \sum_{n \geq 0} q^{4n^2 + 4n + 1} a_{2n+1}'(q^2) \Biggr) \\
& = \frac{2(q; q^2)_{\infty}}{(q^2; q^2)_{\infty}} \Biggl( \sum_{n \geq 0} q^{12n^2 - 2n} \sum_{j=-n}^{n-1} q^{-4j^2 - 4j} -  \sum_{n \geq 0} q^{12n^2 + 2n} \sum_{j=-n}^{n-1} q^{-4j^2 - 4j} \\
& +  \sum_{n \geq 0} q^{12n^2 + 10n + 3} \sum_{j=-n}^{n} q^{-4j^2} - \sum_{n \geq 0} q^{12n^2 + 14n + 5} \sum_{j=-n}^{n} q^{-4j^2} \Biggr). \\
\end{aligned}
\end{equation*}

\noindent So,

\begin{equation*}
\begin{aligned}
\mathcal{W}_{3}(q) & = \frac{2(q; q^2)_{\infty}}{(q^2; q^2)_{\infty}} \Biggl( \Bigl(  \sum_{\substack{r, s \geq 0 \\ r \not\equiv s \imod{2} }} - \sum_{\substack{r, s < 0 \\ r \not\equiv {s \imod{2}} }} \Bigr) q^{2r^2 + 8rs + 5r + 2s^2 + 5s + 3} \\
& + \Bigl(  \sum_{\substack{r, s \geq 0 \\ r \equiv s \imod{2} }} - \sum_{\substack{r, s < 0 \\ r \equiv {s \imod{2}} }} \Bigr) q^{2r^2 + 8rs + 5r + 2s^2 + 5s + 3}  \Biggr) \\
& =  \frac{2(q; q^2)_{\infty}}{(q^2; q^2)_{\infty}} \Biggl( \Bigl(  \sum_{r, s \geq 0} - \sum_{r, s < 0} \Bigr) q^{2r^2 + 8rs + 5r + 2s^2 + 5s + 3}  \Biggr) \\
& = \frac{2q^3 (q; q^2)_{\infty}}{(q^2; q^2)_{\infty}} f_{1,2,1}(-q^7, -q^7, q^4). 
\end{aligned}
\end{equation*}

\noindent By Theorem \ref{hm}, (\ref{g}), (\ref{j1}) and (\ref{j2}), we have

\begin{equation*}
\begin{aligned}
f_{1,2,1}(-q^7, -q^7, q^4) = 2q^{-3} j(-q, q^4) m(-q, q^{12}, -1) + \theta_{1,1}(-q^7, -q^7, q^4)
\end{aligned}
\end{equation*}

\noindent and so 

\begin{equation*} 
\begin{aligned}
\mathcal{W}_{3}(q) = 4m(-q, q^{12}, -1) + \frac{2q^3 \theta_{1,1}(-q^7, -q^7, q^4)}{j(-q, q^4)}.
\end{aligned}
\end{equation*}

Finally, for (\ref{m6}), apply Lemma \ref{lemma2} and let $b$, $c \to \infty$ in (\ref{limitBailey}) to get

\begin{equation*}
\begin{aligned}
\mathcal{W}_{4}(q) & = \sum_{n \geq 0} q^{n^2 + n} \beta_{n}'(q) \\
& = \frac{(1-q)}{(q)_{\infty}} \sum_{n \geq 0} q^{n^2 + n} \alpha_{n}'(q) \\
& = \frac{(1-q)}{(q)_{\infty}}\Biggl( \sum_{n \geq 0} q^{4n^2 + 2n} \alpha_{2n}'(q) + \sum_{n \geq 0} q^{4n^2 + 6n + 2} \alpha_{2n+1}'(q) \Biggr) \\
& =  \frac{1}{(q)_{\infty}} \Biggl( \sum_{n \geq 0} q^{8n^2 + 5n} \sum_{j=-n}^{n-1} q^{-2j^2 - 2j} + \sum_{n \geq 0} q^{8n^2 + 3n} \sum_{j=-n}^{n} q^{-2j^2} \\
& - \sum_{n \geq 0} q^{8n^2 + 13n + 5} \sum_{j=-n}^{n} q^{-2j^2} - \sum_{n \geq 0} q^{8n^2 + 11n + 3} \sum_{j=-n-1}^{n} q^{-2j^2 - 2j} \Biggr). \\
\end{aligned}
\end{equation*}

\noindent After replacing $n$ with $-n-1$ in the third and fourth sums, we let $n=(r+s+1)/2$, $j=(r-s-1)/2$ in the first and fourth sums and $n=(r+s)/2$, $j=(r-s)/2$ in the second and third sums to get

\begin{equation*}
\begin{aligned}
\mathcal{W}_{4}(q) & = \frac{1}{(q)_{\infty}} \Biggl(  \Bigl(  \sum_{\substack{r, s \geq 0 \\ r \not\equiv s \imod{2} }} - \sum_{\substack{r, s < 0 \\ r \not\equiv {s \imod{2}} }} \Bigr) q^{\frac{3}{2}r^2 + 5rs + \frac{13}{2}r + \frac{3}{2}s^2 + \frac{13}{2}s + 5} \\
& + \Bigl(  \sum_{\substack{r, s \geq 0 \\ r \equiv s \imod{2} }} - \sum_{\substack{r, s < 0 \\ r \equiv {s \imod{2}} }} \Bigr) q^{\frac{3}{2}r^2 + 5rs + \frac{3}{2}r + \frac{3}{2}s^2 + \frac{3}{2}s}  \Biggr) \\
& = \frac{1}{(q)_{\infty}} \Biggl( 2q^{13} f_{3,5,3}(-q^{25}, -q^{29}, q^4) + f_{3,5,3}(-q^9, -q^9, q^4) + q^{11} f_{3,5,3}(-q^{25}, -q^{25}, q^4) \Biggr) \\
& = \frac{1}{(q)_{\infty}} f_{3,5,3}(q^3, q^3,q)
\end{aligned}
\end{equation*}

\noindent where in the last step we have used (\ref{fprop1}) and (\ref{fprop2}). By Theorem \ref{hm}, (\ref{g}), (\ref{j1}), (\ref{j2}) and (\ref{appell1}), we have

\begin{equation*}
\begin{aligned}
f_{3,5,3}(q^3, q^3, q) = -2q^{-4} j(q,q^3) m(-q^5, q^{48}, -1) - 2q^{-2} j(q, q^3) m(-q^{11}, q^{48}, -1) + \theta_{3,2}(q^3, q^3, q)
\end{aligned}
\end{equation*}

\noindent and so 

\begin{equation*} 
\begin{aligned}
\mathcal{W}_{4}(q)= -2q^{-4} m(-q^5, q^{48}, -1) - 2q^{-2} m(-q^{11}, q^{48}, -1) + \frac{\theta_{3,2}(q^3, q^3, q)}{j(q, q^3)}.
\end{aligned}
\end{equation*}

\end{proof}

\begin{proof}[Proof of Corollary \ref{8th}]
Equations (4.36) and (4.38) of \cite{Hi-Mo1} state

$$
S_{1}(q) = -2q^{-1} m(-q, q^8, -1) + \frac{\overline{J}_{3,8} J_{2,8}^2}{q J_{1,8}^2}
$$

\noindent

$$
T_{1}(q) = q^{-1} m(-q, q^8, q^6).
$$

\noindent By (\ref{m2}), (\ref{j2}) and (\ref{appell2}), the claim is equivalent to the identity

$$
 \frac{ \overline{J}_{3,8} J_{2,8}^2}{ J_{1,8}^2} + \frac{2q \theta_{1,2}(-q^2, -q^2, q)}{j(-1,q)} = \frac{-2j(q^8, q^{24})^3 j(-q^6, q^8)}{j(q^6, q^8) j(-1, q^8) j(-q^7, q^8)}.
$$
We have verified this identity using Garvan's MAPLE program (see \url{http://www.math.ufl.edu/~fgarvan/qmaple/theta-supplement/}).

\end{proof}

\section*{Acknowledgements}
The authors would like to thank Song Heng Chan and Renrong Mao for their helpful comments.


\begin{thebibliography}{99}

\bibitem{An1.5}
G.E. Andrews, \emph{The fifth and seventh order mock theta functions}, Trans. Amer. Math. Soc. {\bf 293} (1986), no. 1, 113--134.

\bibitem{An2} 
G.E. Andrews, \emph{$q$-Series: Their Development and Application in Analysis, Number
Theory, Combinatorics, Physics, and Computer Algebra}, volume 66 of Regional
Conference Series in Mathematics. American Mathematical Society, Providence, RI,
1986.

\bibitem{An3}
G.E. Andrews, \emph{Bailey's transform, lemma, chains and tree}, Special functions 2000: current perspective and future directions (Tempe, AZ), 1--22, NATO Sci. Ser. II Math. Phys. Chem., {\bf 30}, Kluwer Acad. Publ., Dordrecht, 2001. 

\bibitem{An4}
G.E. Andrews, \emph{Partitions, Durfee symbols, and the Atkin-Garvan moments of ranks}, Invent. Math. {\bf 169} (2007), 37--73.

\bibitem{AH}
G.E. Andrews and D. Hickerson, \emph{Ramanujan's ``lost" notebook. VII. The sixth order mock theta functions}, Adv. Math. {\bf 89} (1991), no. 1, 60--105.


\bibitem{bk}
K. Bringmann and B. Kane, \emph{Multiplicative $q$-hypergeometric series arising from real quadratic fields}, Trans. Amer. Math. Soc. {\bf 363} (2011), no. 4, 2191--2209.

\bibitem{Br-Lo-Os1}
K. Bringmann, J. Lovejoy and R. Osburn, \emph{Automorphic properties of generating functions for generalized rank moments and Durfee symbols}, Int. Math. Res. Not. (2010), no. 2, 238--260.  

\bibitem{Br-On1}
K. Bringmann and K. Ono, \emph{The $f(q)$ mock theta function conjecture and partition ranks}, Invent. Math. {\bf 165} (2006), 243--266. 

\bibitem{Br-On2}
K. Bringmann and K. Ono, \emph{Dyson's ranks and Maass forms}, Ann. of Math. {\bf 171} (2010), 419--449.

\bibitem{Go-Mc1}
B. Gordon and R.J. McIntosh, \emph{Some eighth order mock theta functions}, J. London Math. Soc. (2) {\bf 62} (2000), 321--335. 

\bibitem{Hi-Mo1} 
D. Hickerson and E. Mortenson, \emph{Hecke-type double sums, Appell-Lerch sums, and mock theta functions (I)}, preprint available at \url{http://arxiv.org/abs/1208.1421}

\bibitem{On1}
K. Ono, \emph{Unearthing the visions of a master: harmonic Maass forms and number theory},
Proceedings of the 2008 Harvard-MIT Current Developments in Mathematics Conference, International Press, Somerville, MA, 2009, 347--454.

\bibitem{war}
S. O. Warnaar, \emph{50 years of Bailey's lemma}, Algebraic combinatorics and applications (G{\"o}{\ss}weinstein, 1999), 333--347, Springer, Berlin, 2001.

\bibitem{Za1}
D. Zagier, \emph{Ramanujan's mock theta functions and their applications (after Zwegers and Ono-Bringmann)}, Ast\'erisque {\bf 326} (2009), 143--164.

\bibitem{Zw1}
S. Zwegers, \emph{Mock Theta Functions}, PhD Thesis, Universiteit Utrecht (2002).

\end{thebibliography}
\end{document}